\documentclass[12pt,a4paper]{amsart}
\usepackage[cp1250]{inputenc}
\usepackage{graphicx}
\usepackage{tikz}
\usepackage{enumerate}
\usepackage{rotating}
\usepackage{soul}
\usepackage[OT4]{fontenc} 

\theoremstyle{definition}

\author{Zbigniew Burdak, Marek Kosiek, Patryk Pagacz, Marek S{\l}oci{\'n}ski}
\title{Invariant subspaces of $\h^2(\T^2)$ and $L^2(\T^2)$ preserving compatibility}

\address{Zbigniew Burdak, Department of Applied Mathematics, University of
Agriculture, ul. Balicka 253c, 30-198 Krak\' ow,
Poland.}

\email{rmburdak@cyf-kr.edu.pl}

\address{Marek Kosiek, Wydzia\l{} Matematyki i Informatyki,
Uniwersytet Jagiello\'nski, ul. Prof. St. \L{}ojasiewicza 6, 30-348 Krak\'ow, Poland
}\email{Marek.Kosiek@im.uj.edu.pl}

\address{Patryk Pagacz, Wydzia\l{} Matematyki i Informatyki,
Uniwersytet Jagiello\'nski, ul. Prof. St. \L{}ojasiewicza 6, 30-348 Krak\'ow, Poland
 }\email{Patryk.Pagacz@im.uj.edu.pl}

\address{Marek S\l{}oci\'nski, Wydzia\l{} Matematyki i Informatyki,
Uniwersytet Jagiello\'nski, ul. Prof. St. \L{}ojasiewicza 6, 30-348 Krak\'ow, Poland
 }\email{Marek.Slocinski@im.uj.edu.pl}

\keywords{Invariant subspaces, Beurling theorem, multiplication operator over bi-disk, Hardy space}

\subjclass{Primary 47A15; Secondary 47B37}

\thanks{Research was supported by the Ministry of Science and Higher Education of the
Republic of Poland}

\newcommand{\ran}{\mathrm{ran}\,}
\newcommand{\h}{\ensuremath{\mathcal H}}

\newcommand{\M}{\ensuremath{\mathcal M}}
\newcommand{\B}{\ensuremath{\mathcal B}}
\newcommand{\Ll}{\ensuremath{\mathcal L}}
\newcommand{\K}{\ensuremath{\mathcal K}}

\newcommand{\Z}{\ensuremath{\mathbb Z}}
\newcommand{\T}{\ensuremath{\mathbb T}}

\newcommand{\Cplx}{\ensuremath{\mathbb C}}
\usepackage{amssymb}

\theoremstyle{plain}
\newtheorem{thm}{Theorem}[section]

\newtheorem{lem}[thm]{Lemma}
\newtheorem{cor}[thm]{Corollary}
\newtheorem{prop}[thm]{Proposition}

\theoremstyle{definition}
\newtheorem{Def}[thm]{Definition}
\newtheorem{rem}[thm]{Remark}
\newtheorem{ex}[thm]{Example}

\def\Arg{\operatorname{Arg}}
\def\atan{\operatorname{atan}}
\def\Span#1{\bigvee\{#1\}}

\def\C1{C_{1,\cdot}}

\usepackage{color}

\begin{document}

\maketitle

\begin{abstract}
Operators of multiplication by independent variables on the space of square summable functions over the torus and its Hardy subspace are considered. Invariant subspaces where the operators are compatible are described.
\end{abstract}

\section{Introduction}

Let $\B(H)$ be an algebra of bounded linear operators on a complex Hilbert space $H$. The restriction of an operator to an invariant subspace is called a part of the operator and similarly for systems of operators. A subspace $H_0\subset H$ is reducing under an
operator  if and only if $P_{H_0}$ (the orthogonal projection onto $H_0$) commutes with the operator.  An invariant subspace which do not contain any nontrivial, reducing subspace is called purely invariant.
Recall the classical Wold's result \cite{W}.

\begin{thm}\label{W}
Let $V\in \B(H)$ be an isometry. There is a unique decomposition of $H$ into orthogonal, reducing under $V$ subspaces $H_u$, $H_s$,
such that $V|_{H_u}$ is a unitary operator    and $V|_{H_s}$ is a unilateral shift.
 Moreover,
 \begin{equation}\label{Wold}
 H_u=\bigcap_{n\geq 0}V^nH,\quad H_s=\bigoplus_{n\geq0} V^n(\ker V^*).
 \end{equation} $\hfill\Box$
\end{thm}text
There is no natural extension of Wold's result to a pair in general. However, it holds for doubly commuting pairs \cite{Sl}. Recall that operators $T_1,T_2 \in \B(\h)$ \emph{doubly commute} if and only if they commute and $T_1T^*_2 = T^*_2T_1$.

\begin{thm}\label{wbs}  For any pair of doubly commuting isometries $V_1,V_2$ on $H$ there is a unique decomposition
\begin{equation}H=H_{uu}\oplus H_{us}\oplus H_{su}\oplus H_{ss},\label{RR}\end{equation}
 such that
$H_{uu},\; H_{us},\; H_{su},\; H_{ss}$ reduce $V_1$ and $V_2$ and
\begin{enumerate}
\item[] $V_1|_{H_{uu}}, V_2|_{H_{uu}}$ are unitary operators,
\item[] $V_1|_{H_{us}}$ is a unitary operator, $V_2|_{H_{us}}$ is
a unilateral shift,
\item[] $V_1|_{H_{su}}$ is a unilateral shift, $V_2|_{H_{su}}$ is a unitary operator,
\item[] $V_1|_{H_{ss}},V_2|_{H_{ss}}$  are unilateral shifts.
\end{enumerate}
\end{thm}$\hfill\Box$

Let $\T\subset\Cplx, \T^2\subset\Cplx^2$ denote the unit circle, the torus respectively, $L^2(\T), L^2(\T^2)$ the spaces of square summable functions with normalized Lebesgue measure and $\h^2(\T), \h^2(\T^2)$ the respective Hardy spaces. Further, $L^2(\T,H)$ denotes the space of square summable functions over $\T$ valued in the Hilbert space $H$. Recall that  $L^2(\T,H)\simeq L^2(\T)\otimes H$ and $L^2(\T^2)\simeq L^2(\T,L^2(\T))\simeq L^2(\T)\otimes L^2(\T)$.
The operators of multiplication by independent variable(s) are denoted  by $L_zf(z):=zf(z)$ on $L^2(\T)$ and $T_z:=L_z|_{\h^2(\T)}$ and $L_w, L_z, T_w, T_z$ in the case of spaces over the torus. Whenever it is considered invariant or reducing subspace of $\h^2(\cdot), L^2(\cdot)$ without indicating opertor(s) it is assumed to be reducing or invariant under respective multiplication operator(s) where $\cdot$ stands for the circle, the torus or a Hilbert space valued case. Recall that the function $\psi\in L^\infty(\T)$ is unimodular if $|\psi(z)|=1$ for almost every $z\in\T$ and similarly on the torus. By the result of Helson in \cite{Helson} any reducing subspace of $L^2(\T)$ is of the form $\chi_\delta L^2(\T)$, for some Borel set $\delta\subset\T$, while purely invariant subspace is of the form $\psi\h^2(\T)$ for $\psi$ a unimodular function. A similar result on the torus, but only for reducing subspaces  was obtained in \cite{GM}, Lemma 3. Hardy spaces do not contain nontrivial reducing subspaces. Indeed, $T_z\in\h^2(\T)$ is a model of a unilateral shift of multiplicity one which do not have reducing subspaces. The proof for the Hardy space over the torus  is given in Section 3. The invariant subspaces of $\h^2(\T)$ are described by inner functions. The function $\phi\in \h^\infty(\T)$ is called inner if $|\phi(z)|=1$ for almost every $z\in\T$.

\begin{thm}[Beurling\cite{Beur}]\label{beur}
Each invariant under $T_z\in \B(\h^2(\T))$ subspace is of the form $\phi \h^2(\T)$, where $\phi$ is an inner function.
\end{thm}

 By the results in \cite{Sl, Sarkar} a model of $n$ tuple of doubly commuting unilateral shifts are operators of multiplication by independent variables on the Hardy space over the polydisk $\T^n$. Note that $n$ tuple of operators doubly commute only when each pair of different operators in the $n$ tuple doubly commute. Thus in the case $n=1$ doubly commutativity is vacuously satisfied and the model describes any unilateral shift of multiplicity one. From such point of view a  generalization of Beurling theorem to $n$ tuple is that inner functions describe invariant subspaces of doubly commuting unilateral shifts where the operators preserve doubly commutativity. Such a generalization is precisely formulated and proved in \cite{SSW}. Let us only point out that it covers the classical Beurling Theorem (with its generalizations by Lax and by Halmos) as well as the following result of Mandrekar \cite{Man}.
\begin{thm}\label{Man}
Let $T_w, T_z\in \B(\h^2(\T^2))$ be multiplications by independent variables $w,z$, respectively.
Any invariant under  $T_w, T_z$ subspace $\M\not=\{0\}$ is of the form $\phi\h^2(\T^2)$, with $\phi$ being inner function if and only if $T_w, T_z$ doubly commute on $\M$.
\end{thm}

Obviously, there are other subspaces invariant under the considered pair, where respective restrictions are no longer doubly commuting. An example is $\mathcal{M}:=\h^2(\T^2)\ominus \Cplx \cdot \mathbf{1}$ (orthogonal to constant functions). Then $(T_w|_{\mathcal{M}})^*T_z|_{\mathcal{M}}w=(T_w|_{\mathcal{M}})^*zw=z$ while $T_z|_{\mathcal{M}}(T_w|_{\mathcal{M}})^*w=0$. The invariant subspaces of $\h^2(\T^2)$ and  $L^2(\T^2)$ has been investigated in \cite{GS_inv} with respect to the Wold-type decomposition showed in \cite{GS_decom}.
  The aim of the paper is to improve the characterization of invariant subspaces of $\h^2(\T^2)$ and $L^2(\T^2)$. We take advantage of the concept of compatible isometries which covers the mentioned results as well as many examples. Precisely we describe invariant subspaces  where the operators preserve compatibility.
Section 2 is devoted to compatible pairs of isometries.
Section 3 concerns $\h^2(\T^2)$ where the main Theorem \ref{mainH} generalize Theorem \ref{Man} as well as results from \cite{GS_inv}. It is compared with relatively recent results in \cite{Seto2}.
Section 4 concerns $L^2(\T^2)$ where the main result is Theorem \ref{mainL}. There are constructed also unitary extensions of parts of $L_w, L_z$ obtained by each type of invariant subspace. In particular an example of a proper subspace of $L^2(\T^2)$ reducing $L_w, L_z$ to bilateral shifts is obtained.

\section{Compatible pairs of isometries}
 A pair of commuting isometries $V_1,V_2$ is said to be \emph{compatible} if $P_{\ran(V_1^m)}$ commute with $P_{\ran(V_2^n)}$ for every $m,n \in\mathbb{Z}_+$ (see \cite{Mucomp, Mucomp2}).  Compatible isometries can be decomposed into doubly commuting pairs, pairs given by diagrams and generalized powers (see \cite{Comp}). Let us recall the definitions of the above classes of operators.

 The idea of pairs given by a diagram appeared in \cite{Mucomp} (Example 1) while the precise definition and classification of diagrams can be found in  \cite{BKPSLAA}.
\begin{Def}
A set $J\subset\Z^2$ is \emph{a diagram} if $J+\Z_+^2\subset J$ as addition the coordinates.

Diagrams $J, J'$ are translation equivalent if $J=(i,j)+J'$ for some $(i,j)\in\mathbb{Z}$.

Diagrams equivalent to $\mathbb{Z}^2, \mathbb{Z}_+^2, \mathbb{Z}\times\mathbb{Z}_+, \mathbb{Z}_+\times\mathbb{Z}$ are called  \emph{simple}.

A diagram $J$ is \emph{periodic} if there are positive numbers $m,n$ such that for
 $J_0:=(\{0,1,\dots,m-1\}\times \mathbb{Z})\cap J$ and $J_k=J_0+k(m,-n)$ it holds
 $J=\bigcup_{k\in\mathbb Z} J_k,$  where $J_k$ are pairwise disjoint.
 The set $J_0$ is called a period of the diagram $J$. Moreover, $J_0$ and the positive integers $m,n$ determine $J$ by the formula $J=\bigcup_{k\in\mathbb{Z}}J_0+k(m,-n)$.

 Simple and periodic diagrams are called \emph{regular}, remaining are \emph{irregular}.
 \end{Def}
 For purposes of the paper it is convenient to define a pair of isometries given by an arbitrary diagram $J$ by giving the model based on $L_w, L_z$. Therefore, instead of Definition 4.3 from \cite{BKPSLAA} it is given another one. First step is to reformulate Definition 4.4 from \cite{BKPSLAA} as follows:

 \begin{Def}\label{defdiagram}
 A simple pair of isometries given by a diagram $J$ is a pair unitarily equivalent to operators $L_w|_{\M_J}, L_z|_{\M_J}$ where $\mathcal{M}_J:=\Span{w^iz^j:(i,j)\in J}\subset L^2(\T^2)$.
  \end{Def}

    Note that $\M_{\mathbb{Z}_+^2}=\h^2(\T^2)$ and the respective restrictions are $T_w, T_z$. Since $\ker T_w^*=\Span{z^i, i\in\Z_+}$ and $\ker T_z^*=\Span{w^i, i\in\Z_+}$ then the operators are unilateral shifts of infinite multiplicity. However $T_w, T_z$, as a pair, are generated by $\ker T_w^*\cap\ker T_z^*$ which is one-dimensional. Following this idea the multiplicity of a pair of doubly commuting unilateral shifts is defined as the dimension of $\ker T_w^*\cap\ker T_z^*$. In the case of a pair given by an arbitrary diagram $J$ the definition of multiplicity is not so obvious. More useful is the other approach. Recall that a model of doubly commuting unilateral shifts of multiplicity $n$ is obtained on the Hardy subspace of $L^2(\T^2,H)$ where $\dim H$ equals to the multiplicity (see f.e. \cite{NF,RR,Sarkar}). Any function in $L^2(\T^2,H)$ is identified with its Fourier series where coefficients are in $H$. Thus it is an easy observation that for any subspace $H_0\subset H$ the space $L^2(\T^2,H_0)$ may be regarded as a reducing subspace of $L^2(\T^2,H)$. Consequently, a decomposition $H=H_0\oplus H_1$ generates the reducing decomposition $L^2(\T^2,H)=L^2(\T^2,H_0)\oplus L^2(\T^2,H_1)$. Moreover, if $\M_{JH}\subset L^2(\T^2,H)$ denotes the subspace of functions whose Fourier coefficients with indices out of $J$ are zero then also $\M_{JH}=\M_{JH_0}\oplus \M_{JH_1}$. On the other hand if $\dim H=1$ then multiplication operators on $L^2(\T^2,H)$ are unitarily equivalent to $L^2(\T^2)$ and their restrictions to $\M_{JH}$ are simple pairs as in Definition \ref{defdiagram}.  Eventually, consider the restrictions of $L_w, L_z\in L^2(\T^2,H)$ to $\M_{JH}$ and $\{v_1,v_2,\dots\}$ an orthnormal basis of $H$. Then $H=\bigoplus_{n\ge 0}\mathbb{C}v_n$ generates the decomposition $\M_{JH}=\bigoplus_{n\ge 0}\M_{J\mathbb{C}v_n}$. In other words a pair $L_w|_{\M_{JH}}, L_z|_{\M_{JH}}$ is decomposed into $N$ simple pairs of isometries given by a diagram $J$ where $N=\dim H$. Since the number of simple pairs equals to the dimension of $H$ it is unique. Moreover, by Remark 4.5 in \cite{BKPSLAA} any pair of isometries given by a diagram $J$ can be decomposed into simple pairs. Thus $L_w|_{\M_{JH}}, L_z|_{\M_{JH}}$ turns out to be a model of pairs of isometries defined by a diagram.
  \begin{Def}
  A pair of isometries given by a diagram $J$ is a pair unitarily equivalent to $L_w|_{\M_J}, L_z|_{\M_J}\in L^2(\T^2,H)$ where $\M_J=\{f\in L^2(\T^2,H):\hat{f}_{ij}=0, \text{ for } (i,j)\notin J\}$ and $H$ is a Hilbert space. The dimension of $H$ is called the multiplicity of such a pair.
  \end{Def}
Note that the space of a unitary extension of a simple pair given by any diagram is equivalent to the whole $L^2(\T^2)$. Thus a pair given by a diagram may be a part of $L_w, L_z\in L^2(\T^2)$ only if it is a simple pair.
 Let us recall the definition of generalized powers \cite{Comp}.

\begin{Def}\label{gpdef}
Let be given a periodic diagram $J=\bigcup_{k\in\mathbb{Z}}J_0+k(m,-n)$ and a unitary operator $\mathcal{U}\in \B(\mathcal{H})$ with a star cyclic vector $e$.

Define
 \begin{itemize}
 \item $H:=\bigoplus_{(i,j)\in J_0}H_{i,j}$ where $H_{i,j}=\mathcal{H}$,
 \item $e_{i,j}\in H$ a vector such that $P_{H_{i_0,j_0}}e_{i,j}=\mathcal{U}^ke$ and $P_{H_{i_0',j_0'}}e_{i,j}=0$ for $(i_0',j_0')\ne(i_0,j_0)$ for every $(i,j)\in J$ where $(i_0, j_0)\in  J_0$ and $k\in\mathbb{Z}$ are unique such that $(i,j)=(i_0+km,j_0-kn)$,
 \item $V_1 e_{i,j}=e_{i+1,j}$  and $V_2e_{i,j}=e_{i,j+1}$.
 \end{itemize}

 A pair of isometries $V_1, V_2$ is called \emph{generalized powers}.
 \end{Def}
  Generalized powers are pairs of compatible unilateral shifts $V_1, V_2\in\B(H)$ commuting with $U$ and such that $V_1^m=UV_2^n$ where $U$ denotes the extension of $\mathcal{U}$ to $H$ given in a natural way $U(\sum_{(i,j)\in J_0} x_{i,j})=\sum_{(i,j)\in J_0}\mathcal{U}x_{i,j,}$ (see \cite{Comp,BKPSLAA}).

A period $J_0$ an numbers $m,n$ determine a periodic diagram. However, as in other periodic concepts, the same diagram can be denoted by various periods.
  \begin{rem}\label{periodre}
 Let $V_1, V_2$ be a pair of generalized powers given by a diagram $J=\bigcup_{k\in\mathbb{Z}}J_0+k(m,-n)$. If we define a multiplied period $J'_0:= \bigcup_{k=0}^{l-1}J_0+k(m,-n)$ then $J=\bigcup_{k\in\mathbb{Z}}J'_0+k(lm,-ln)$. Since $U=V_2^{*n}V_1^m$ is unitary then one can show that  $U^l=(V_2^{*n}V_1^m)^l=V_2^{*ln}V_1^{lm}=U'$. Moreover, $H$ is decomposed into larger number of spaces $H'_{i,j}$, so they are different than $H_{i,j}$. One can check that $H_{i,j}=\bigoplus_{k=0}^{l-1}H'_{i+km,j-kn}$. In conclusion, generalized powers depend on a periodic diagram, not on its period. However, the unitary operator $\mathcal{U}$ is related to the choice of a period.
 \end{rem}

 The next result gives an equivalent condition for a pair to be generalized powers.

  \begin{lem}\label{gplemma}
   Isometries  $V_1,V_2$ are generalized powers given by a periodic diagram $J=\bigcup_{k\in\mathbb{Z}}J_0+k(m,-n)$ if and only if there is a decomposition $H:=\bigoplus_{(i,j)\in J_0}H_{i,j}$ such that:
 \begin{itemize}
      \item $V_2^{*n}V_1^m|_{H_{i,j}}$ are unitary operators on $H_{i,j}$ having a star cyclic vector,
    \item $V_1^{i-i'}V_2^{j-j'}$ are unitary operators between $H_{i,j}$ and $H_{i',j'}$ where $V_\iota^\kappa=V_\iota^{*|\kappa|}$ for $\kappa<0$.
      \end{itemize}
      \end{lem}
      \begin{proof}
   Assume $V_1, V_2$ to be generalized powers. The condition $V_1^m=UV_2^n$ is equivalent to $V_2^{*n}V_1^m=U$ which simply means that $V_2^{*n}V_1^m$ is a unitary operator. Moreover, the formula $U(\sum_{(i,j)\in J_0} x_{i,j})=\sum_{(i,j)\in J_0}\mathcal{U}x_{i,j}$ implies that $H_{i,j}$ reduces $V_2^{*n}V_1^m$ for any $(i,j)\in J_0$. Since $V_2 H_{i,j}=H_{i,j+1}$  for $(i,j)\in J_0$ and $V_1H_{i,j}=H_{i+1,j}$ for $(i,j)\in J_0$ where  $i\ne m-1$ then $V_1^{i-i'}V_2^{j-j'}$, as an operator acting between $H_{i,j},$ and $H_{i',j'}$ is unitary.

   For the reverse implication note that the powers $m,n$ in the first condition and the set $J_0$ in the decomposition provides a periodic diagram $J$.
    By the first condition $U:=V_2^{*n}V_1^m$ is unitary. However, since $V_1, V_2$ are isometries it follows $\ran(V_1^m)=\ran(V_2^n)$. Hence $V_1^m=V_2^nV_2^{*n}V_1^m=V_2^nU$ and $V_2U=V_2V_2^{*n}V_1^m=V_2^{*n-1}V_1^m=V_2^{*n}V_2V_1^m=V_2^{*n}V_1^mV_2=UV_2$. Thus $U$ commutes with $V_2$. Similarly $U^*$ commutes with $V_1$ and, as a unitary operator, doubly commute with $V_1$. Thus $U$ (doubly) commutes with both $V_1, V_2$. Since $H_{i,j}$ reduces $U$ for each $(i,j)\in J_0$, the operator $U|_{H_{i,j}}$ is unitary. Moreover, the proved commutativity implies the equivalence of all $U|_{H_{i,j}}$, in details $U|_{H_{i,j}}=(V_1^{i-i'}V_2^{j-j'}|_{H_{i,j}})^*U|_{H_{i',j'}}(V_1^{i-i'}V_2^{j-j'}|_{H_{i,j}})$. So, $\mathcal{U}:=U|_{H_{i,j}}$ does not depend on the choice of $H_{i,j}$. Moreover, $U(\sum_{(i,j)\in J_0} x_{i,j})=\sum_{(i,j)\in J_0}\mathcal{U}x_{i,j}$. Now we can follow Definition \ref{gpdef} and define vectors $e_{i,j}$. Note that $Ue_{i,j}=e_{i+m,j-n}$. It has left to check formulas $V_1 e_{i,j}=e_{i+1,j}$  and $V_2e_{i,j}=e_{i,j+1}$ for $(i,j)\in J$.  By the commutativity of $U$ with $V_1, V_2$ and the relation $Ue_{i,j}=e_{i+m,j-n}$  it is enough to prove the formulas for $(i,j)\in J_0$. By the second condition for $(i',j')=(i,j+1)$ and $(i',j')=(i+1,j)$ respectively we get $V_2e_{i,j}=e_{i,j+1}$ for $(i,j)\in J_0$ and $V_1e_{i,j}=e_{i+1,j}$ for $(i,j)\in J_0, i\ne m-1$ while $V_1e_{m-1,j}=V_1^me_{0,j}=V_2^nUe_{0,j}=V_2^ne_{m,j-n}=e_{m,j}$.  \end{proof}
  Note that if for some $(i,j),(i',j')\in J$ holds $i_0'=i_0,j_0'=j_0$ for respective $(i_0,j_0),(i_0',j_0')\in J_0$ then vectors $e_{i,j},e_{i',j'}$ may not be orthogonal to each other (f.e. if $\mathcal{U}=I$ they are equal). This differs generalized powers from pairs given by diagrams. However, if $\mathcal{U}$ is a bilateral shift then the considered vectors are orthogonal. In fact generalized powers defined by bilateral shifts are precisely pairs defined by periodic diagrams. Moreover, two equivalent diagrams define unitarily equivalent pairs of isometries and consequently simple diagrams define doubly commuting pairs of isometries. The uniqueness in the following decomposition (Theorem 4.12 from \cite{BKPSLAA}) follows by the irregularity of diagrams in $H_d$.

\begin{thm} \label{decomposition_c_cnc}
For any pair of commuting isometries $V_1, V_2$ on the Hilbert space $H$ there is a unique decomposition: $$H = H_{uu}\oplus H_{us}\oplus H_{su}\oplus H_{Hardy}\oplus H_{d}\oplus \h_{gp} \oplus \h_{cnc}$$ where the restrictions of the operators $V_1, V_2$ to the space:
\begin{enumerate}
\item  $H_{uu}$ are unitary,
\item  $H_{su}$ are a unilateral shift and a unitary operator respectively,
\item  $H_{us}$ are a unitary operator and a unilateral shift respectively,
\item  $H_{Hardy}$ are a pair of doubly commuting unilateral shifts,
\item  $H_{d}$ can be decomposed into pairs given by irregular diagrams,
\item  $H_{gp}$ can be decomposed into generalized powers,
\item  $H_{cnc}$ is a completely non compatible pair.
\end{enumerate}
\end{thm}

\section{Invariant subspace on Hardy space}
In the introduction it was recalled that $\h^2(\T^2)$ has no nontrivial reducing subspaces. Indeed, since $T_w, T_z$ doubly commute they are compatible and projections $(I-T_w^*T_w), (I-T_z^*T_z)$ commute. Thus $(I-T_w^*T_w)(I-T_z^*T_z)=P_{\ker T_w^*\cap\ker T_z^*}$ and every subspace reducing under $T_w, T_z$ is invariant under $P_{\ker T_w^*\cap\ker T_z^*}$. However, since $\ker T_w^*\cap \ker T_z^*$ is one-dimensional then either the considered subspace contains $\ker T_w^*\cap \ker T_z^*$ or is orthogonal to it. Since a subspace containing $\ker T_w^*\cap \ker T_z^*$ and invariant under $T_w, T_z$  is the whole $\h^2(\T^2)$ then in the first case the considered subspace is $\h^2(\T^2)$ while in the second it is an orthogonal complement of $\h^2(\T^2)$, so it is a zero subspace.

\begin{thm}\label{mainH}
Let $\{0\}\not=\M \subset \h^2(\T^2)$ be an invariant subspace. The pair $(T_w|_{\M},T_z|_{\M})$ is compatible if and only if $\M=\phi \M_J$, for an inner function $\phi$ and a diagram $J\subset \Z_+^2$, where $\M_J=\Span{w^iz^j:(i,j)\in J}$.
\end{thm}

\begin{proof}
Let $\M=\phi \M_J=\Span{\phi w^iz^j: (i,j)\in J}$ for a given inner function $\phi$ and a diagram $J$. Since $\phi$ is inner then operator $T_\phi:\M_J\ni f\to\phi f\in\phi \M_J$ is unitary with $T_\phi^*=T_{\overline{\phi}}$. Thus a pair $(T_w|_{\M}, T_z|_{\M})$ is unitarily equivalent to $(T_w|_{\M_J}, T_z|_{\M_J})$, so it is given by a diagram $J$.

For the reverse implication let $\M\not=\{0\}$ be such that $T_w|_{\M}, T_z|_{\M}$ are compatible. The pair $T_w|_\M, T_z|_\M$ can be decomposed by Theorem \ref{decomposition_c_cnc}. Since the operators are compatible unilateral shifts, the decomposition is reduced to $\M=\M_{Hardy}\oplus H_d\oplus H_{gp}$. Note that if $\M_{Hardy}\ne \{0\}$ then it is equivalent to $\h^2(\T^2)$ and then $\M_d=\M_{gp}=\{0\}$. Similarly, if $H_d\ne\{0\}$ then it contains a subspace unitarily equivalent to $\h^2(\T^2)$ and $\M_{Hardy}=\M_{gp}=\{0\}$. So we may assume that exactly one subspace in the decomposition is nontrivial. If $T_w|_\M, T_z|_\M$ doubly commute then by Theorem \ref{Man} we get the statement.

For the remaining cases let us show that
\begin{equation}\tag{\dag}\label{vanish}
\lim_{k\to\infty}T_z^{*n_k}T_w^{m_k}x=0,\quad \lim_{k\to\infty}T_w^{*m_k}T_z^{n_k}x=0
\end{equation}
for any $x\in\h^2(\T^2)$ and any increasing sequence $\{(m_k,n_k)\}\subset \Z_+^2$. Since $\h^2(\T^2)=\Span{w^iz^j, i,j\in\Z_+}$, it is enough to show that for any $i,j\in\Z_+$ there is $k$ such that $T_z^{*n_k}T_w^{m_k}w^iz^j=0$. Indeed, since $\ker T_z^*=\Span{w^i:i\in\Z_+}$ and $\{n_k\}$ is increasing, for any vector $w^iz^j$ there is $k$ such that $n_k>j$ and consequently $T_z^{*n_k}T_w^{m_k}w^iz^j=T_z^{*n_k-j}w^{i+m_k}=0$.

 If $T_w|_\M, T_z|_\M$  are generalized powers then there are $m,n\in \Z_+$  such that $T_w^m|_\M=UT_z^n|_\M$, where $U\in\B(\M)$ is a unitary operator commuting with $T_w|_\M$ and $T_z|_\M$. Note that $U=(T_z|_{\M})^{*n}(T_w|_\M)^m=(P_\M T_z^*)^nT_w^m|_\M$. Thus $(P_\M T_z^*)^nT_w^m|_\M$ is an isometry and $P_\M$, as a norm preserving projection, may be removed
 from the formula.  Thus, $U=T_z^{*n}T_w^m|_\M,\;\ran({T_z^{*n}T_w^m|_\M})\subset \M$ and consequently $U^k=(T_z^{*n}T_w^m|_\M)^k=(T_z^{*n}T_w^m)^k|_\M$. Moreover, $T_z^nT_z^{*n}T_w^m|_\M=T_z^nU=UT_z^n|_\M=T_w^m|_\M$ and $U^k=T_z^{*kn}T_z^{kn}U^k=T_z^{*kn}U^kT_z^{kn}|_\M=T_z^{*kn}(T_z^{*n}T_w^m)^kT_z^{kn}|_\M=T_z^{*kn}T_w^{km}|_\M$. Thus from (\ref{vanish}) we get $\|x\|=\|U^kx\|\to 0$ for any $x\in\M$. Since we assumed $\M\ne\{0\}$, the restrictions may not be generalized powers.

 Let $T_w|_\M, T_z|_\M$ be a simple pair given by a diagram $J\subset \Z^2$. Denote $f_{i,j}\simeq w^iz^j$ for $(i,j)\in J$ where $\simeq$ denotes the unitary equivalence as in Definition \ref{defdiagram}.
Let us observe that $J$ is bounded from below i.e. there is $N\in\Z$ such that $J\subset \{(m,n): n \geq N, m\in \Z\}$. Indeed if not, for some $x\in\M$ there is a sequence $\{(m_k,n_k)\}_{k\in\Z_+}$ such that $n_k \geq 1$ and
$$1=\|x\|=\|(P_\M T_z^*)^{n_1}T_w^{m_1}x\|=\|(P_\M T_z^*)^{n_2}T_w^{m_2}(P_\M T_z^*)^{n_1}T_w^{m_1}x\|=\dots.$$
Since all the above expansions preserve norm, then we can omit the projection and rearrange the operators to get
$$1=\|x\|=\|T_z^{*n_k}T_w^{m_k}\dots T_z^{*n_1}T_w^{m_1}x\|=\|T_z^{*n_1+\dots+n_k}T_w^{m_1+\dots+m_k}x\|.$$
However, it contradicts (\ref{vanish}) for the sequence $\{(m_1+\dots+m_k, n_1+\dots+n_k)\}$.
Similarly one can show that the diagram $J$ is bounded from the left by $M$. Thus $J$ is described by a finite sequence $\{(n_\alpha,m_\alpha)\}_{\alpha \in A}$ such that $J=\bigcup\limits_{\alpha\in A}\{(n,m)\in \Z_+^2: n\geq n_\alpha, m\geq m_\alpha\}$ where $A=\{1,2,\dots,K\}$.

\begin{tikzpicture}
[yscale=0.7, xscale=0.95, auto,
kropka/.style={draw=black!50, circle, inner sep=0pt, minimum size=2pt, fill=black!50},
kropa/.style={draw=black, circle, inner sep=0pt, minimum size=4pt, fill=black!100},
kolo/.style={draw=black, circle, inner sep=0pt, minimum size=8pt, semithick},
kwadrat/.style={draw=black, rectangle, inner sep=0pt,  minimum size=10pt, semithick},
romb/.style={draw=black, diamond, inner sep=0pt, minimum size=8pt, semithick},
trojkat/.style ={draw=black, regular polygon, regular polygon sides=3, inner sep=0pt, minimum size=8pt, semithick}]

\draw [black!50] (7,-1) -- (3,-1);
\draw [black!50] (7,0) -- (1,0);
\draw [black!50] (7,1) -- (0,1);
\draw [black!50] (7,2) -- (-2,2);
\draw [black!50] (7,3) -- (-4,3);
\draw [black!50] (7,4) -- (-4,4);
\draw [black!50] (6,-2) -- (6,5);
\draw [black!50] (5,-2) -- (5,5);
\draw [black!50] (4,-2) -- (4,5);
\draw [black!50] (3,-1) -- (3,5);
\draw [black!50] (2,-1) -- (2,5);
\draw [black!50] (1,0) -- (1,5);
\draw [black!50] (0,1) -- (0,5);
\draw [black!50] (-1,1) -- (-1,5);
\draw [black!50] (-2,2) -- (-2,5);
\draw [black!50] (-3,2) -- (-3,5);

\path (6,-1) node [kropka] {} (5,-1) node [kropka] {} (4,-1) node [kropka] {};
\path (6,0) node [kropka] {} (5,0) node [kropka] {} (4,0) node [kropka] {} (3,0) node [kropka] {} (2,0) node [kropka] {};
\path (6,1) node [kropka] {} (5,1) node [kropka] {} (4,1) node [kropka] {} (3,1) node [kropka] {} (2,1) node [kropka] {}
(1,1) node [kropka] {};
\path (6,2) node [kropka] {} (5,2) node [kropka] {} (4,2) node [kropka] {} (3,2) node [kropka] {} (2,2) node [kropka] {}
(1,2) node [kropka] {} (0,2) node [kropka] {} (-1,2) node [kropka] {};
\path (6,3) node [kropka] {} (5,3) node [kropka] {} (4,3) node [kropka] {} (3,3) node [kropka] {} (2,3) node [kropka] {}
(1,3) node [kropka] {} (0,3) node [kropka] {} (-1,3) node [kropka] {} (-2,3) node [kropka] {} (-3,3) node [kropka] {};
\path (6,4) node [kropka] {} (5,4) node [kropka] {} (4,4) node [kropka] {} (3,4) node [kropka] {} (2,4) node [kropka] {}
(1,4) node [kropka] {} (0,4) node [kropka] {} (-1,4) node [kropka] {} (-2,4) node [kropka] {} (-3,4) node [kropka] {};

\draw [black!50]  (7,-2) -- (3,-2) -- (3,-1) -- (1,-1) -- (1,0) -- (0,0) --  (0,1) -- (-2,1) -- (-2,2) -- (-4,2) -- (-4,5);

\path (6,-2) node [kropka] {}  ++(-1,0) node [kropka]  {}
        ++(-1,0) node [kropka] {} ++(-1,0) node [kropa] {}
        ++(0,1) node [kropka] {} ++(-1,0) node [kropka] {}
        ++(-1,0) node [kropa] {} ++(0,1) node [kropka] {}
        ++(-1,0) node [kropa] {}  ++(0,1) node [kropka] {}
        ++(-1,0) node [kropka] {}  ++(-1,0) node [kropa] {}
        ++(0,1) node [kropka] {} ++(-1,0) node [kropka] {} ++(-1,0) node [kropa] {}
        ++(0,1) node [kropka] {} ++(0,1) node [kropka] {};
{\tiny
 \node at (-4.9,0.6) {\begin{rotate}{45}$(m_5,n_5)$\end{rotate}};
 \node at (-2.9,-0.4) {\begin{rotate}{45}$(m_4,n_4)$\end{rotate}};
  \node at (-0.9,-1.3) {\begin{rotate}{45}$(m_3,n_3)$\end{rotate}};
  \node at (0.1,-2.3) {\begin{rotate}{45}$(m_2,n_2)$\end{rotate}};
  \node at (2.1,-3.3) {\begin{rotate}{45}$(m_1,n_1)$\end{rotate}};}
\node at (-4,5.2) {\tiny{M}};
\node at (-3,5.2) {\tiny{M+1}};
\node at (-2,5.2) {\tiny{M+2}};
\node at (-1,5.2) {\tiny{M+3}};
\node at (0,5.2) {\tiny{M+4}};
\node at (1,5.2) {\tiny{M+5}};
\node at (2,5.2) {\tiny{M+6}};
\node at (3,5.2) {\tiny{M+7}};
\node at (4,5.2) {\tiny{M+8}};
\node at (5,5.2) {\tiny{M+9}};
\node at (6,5.2) {\tiny{M+10}};
\node at (7,5.2) {\tiny{i}};

\node at (7.4, 5) {\tiny{j}};
\node at (7.4, 4) {\tiny{N+6}};
\node at (7.4, 3) {\tiny{N+5}};
\node at (7.4, 2) {\tiny{N+4}};
\node at (7.4, 1) {\tiny{N+3}};
\node at (7.4, 0) {\tiny{N+2}};
\node at (7.4, -1) {\tiny{N+1}};
\node at (7.4, -2) {\tiny{N}};

\end{tikzpicture}


Let us denote $\M_{\alpha}:=\Span{f_{i,j}: i\geq m_\alpha, j\geq n_\alpha}$ for $\alpha\in A$.
 Any pair $(T_w|_{\M_\alpha}, T_z|_{\M_\alpha})$ is doubly commuting. Thus, by Theorem \ref{Man}, there is an inner function $\phi_\alpha$ such that $\M_\alpha=\phi_\alpha \h^2(\T^2)$. On the other hand $\M_\alpha=\bigoplus\limits_{\substack{m=0\\n=0}}^\infty T_w^mT_z^n(\Cplx f_{m_\alpha, n_\alpha})$ by \cite{Sl}. Therefore $f_{m_\alpha, n_\alpha}=\phi_\alpha$ for any $\alpha\in A$.
The case $\#A = 1$ is a pair of doubly commuting unilateral shifts and was already considered. If $\#A > 1$, then $n_K>n_1$ and $m_1>m_K$. Note that $$T_z^{n_K-n_1}\M_1=T_w^{m_1-m_K}\M_K=T_z^{n_K-n_\alpha}T_w^{m_1-m_\alpha}\M_\alpha$$
for any $\alpha\in A$. Moreover, $$f_{m_1,n_K}=w^{m_1-m_\alpha}z^{n_K-n_\alpha}\phi_\alpha$$
for any $\alpha\in A$. In particular $$z^{n_K-n_1}\phi_1=w^{m_1-m_K}\phi_K.$$ Comparing Fourier coefficients of both sides we conclude that there exists an inner function $\phi$ such that $$w^{m_1-m_K}\phi=\phi_1\quad z^{n_K-n_1}\phi=\phi_K.$$
Finally, $f_{m_\alpha,n_\alpha}=\phi_\alpha=w^{m_\alpha-m_K}z^{n_\alpha-m_1}\phi$ for any $\alpha\in A$  and consequently $\M=\phi \M_J$.

\end{proof}

 Invariant subspaces $\M$ of $\h^2(\T)$, such that the contractions $T_w^*|_{\h^2(\T^2) \ominus \M},\;T_z^{*}|_{\h^2(\T^2) \ominus \M}$ doubly commute, are considered in \cite{Seto2}. The condition may appeared to be in some relation with compatibility. However, the results are disjoint. Let us recall Theorem 2.1. from \cite{Seto2}.
\begin{thm}
Let $N$ be a backward shift invariant subspace of $\h^2(\T^2)$ and
$N\neq \h^2(\T^2)$. Then $T_wT_z^*=T_z^*T_w$ on $N$ holds if and only if $N$ has one of the following
forms:
\begin{itemize}
\item $N = \h^2(\T^2) \ominus \phi_w\h^2(\T^2)$;
\item  $N = \h^2(\T^2)\ominus \phi_z\h^2(\T^2)$;
\item  $N = (\h^2(\T^2) \ominus \phi_w\h^2(\T^2)) \cap (\h^2(\T^2) \ominus \phi_z\h^2(\T^2))$;
\end{itemize}
where $\{(w,z)\mapsto\phi_w(w)\}$ and $\{(w,z)\mapsto\phi_z(z)\}$ are one variable inner functions.
\end{thm}

 It is clear that subspaces given by a diagram usually do not satisfy the condition of double commutativity on orthogonal complement. The following example shows that the reverse implication  may not hold as well.

\begin{ex}
Let $x = \sum_{j=0}^{\infty} \lambda^j w^j \in \h^2(\T^2)$, for some fixed $\lambda\in\Cplx$, $|\lambda|<1$. Note that $T_w^*x=\lambda x$,
$T_z^* x = 0$. Thus $N := \mathbb{C} x$ is invariant under $T^*_w, T^*_z$ and $\M := \h^2(\T^2)\ominus N$ is invariant under $T_w, T_z$. Denote
$S_w := T_w|_\M, S_z := T_z|_\M \in \B(\M)$. Since $T_z^*|_N = 0$ it doubly commute with $T^*_w|_N$.
However $S_w, S_z$ are not compatible. Indeed, let $y := T_zx \in \M$. Then $S^*_z y =
P_\M T^*_z T_zx = P_\M x = 0$ and so
$S_wS^*_wS_zS^*_z y = 0$. On the other hand $S_zS^*_zS_wS_w^*y =S_zS^*_z (\sum_{j=1}^{\infty} \lambda^j w^j z) = S_zP_\M(\sum_{j=1}^{\infty} \lambda^j w^j ) = S_zP_\M(x - 1) = -S_zP_\M 1 \not= 0$ because the constant function $1$ does not belong to $N$.
\end{ex}

\section{Invariant subspace on $L^2(\T^2)$}

In this section it will be showed that the space $L^2(\T^2)$ contains all the compatible types of invariant subspaces. Precisely, a subspace of each type described in  Theorem \ref{decomposition_c_cnc} may be represented by some invariant subspace of $L^2(\T^2)$. In fact we do not consider the completely non compatible case, but such a subspace may be easily constructed from a diagram type subspace (see Example 5.2 in \cite{BKPSLAA}). Each type is considered in a separate theorem and the results are summarized in Theorem \ref{mainL}. Moreover, a unitary extension of each type is described. Then the coexistence of respective types is investigated by the following observation.

\begin{rem}\label{KbotK}
  Let $V_1, V_2\in\B(\h)$ be a pair of commuting isometries and  $U_1,U_2\in\B(\K)$ be
  its minimal unitary extension. For any reducing decomposition $\h=\h_1\oplus\h_2$ of $V_1,V_2$ it holds $\K=\K_1\oplus \K_2$ where $ \K_i:=\bigvee_{k,l\in \Z} U_1^kU_2^l \h_i \subset \K,$ for $i=\{1,2\}$.
\end{rem}

Since a pair $L_w|_\M, L_z|_\M$ is unitary if and only if $\M$ reduces $L_w, L_z$ then the space of a unitary extension is equal $\chi_\Delta L^2(\T^2)$ for some Borel set $\Delta\subset \T^2$ (\cite{GM}, Lemma 3). Spaces $\chi_\Delta L^2(\T^2)$ are orthogonal if the respective Borel sets are almost disjoint (their common part is of the measure zero).
\begin{thm}\label{diagramL}
Let $\M\not=\{0\}$ be an invariant subspace of $L^2(\T^2)$. The pair $L_w|_{\M},L_z|_{\M}$ is given by a diagram if and only if $$\M=\psi \M_J,$$
where $\M_J:=\Span{w^iz^j:(i,j)\in J}$ and $\{(w,z)\mapsto\psi(w,z)\}$ is a unimodular function. Moreover, the space of a minimal unitary extension of $(L_w|_{\M},L_z|_{\M})$  equals to $L^2(\T^2)$.
\end{thm}
\begin{proof}
 Let $(L_w|_{\M},L_z|_{\M})$ be given by a diagram $J \in \Z^2$ and $\{e_{i,j}\}_{(i,j)\in J}$ be the underlying basis of $\M$. Precisely $e_{i,j}\simeq w^iz^j$, where $\simeq$ is a unitary equivalence as in Definition \ref{defdiagram}. Then $\M=\bigoplus_{(i,j)\in J}\mathbb{C}e_{i,j}$ and $L_w e_{i,j}=e_{i+1,j}, L_ze_{i,j}=e_{i,j+1}$. Note that the set 
 $\{(n,m)\in J: (n-1,m)\not\in J \textnormal{ and } (n,m-1)\not\in J\}$ can be ordered in a sequence $(n_\alpha,m_\alpha)_{\alpha\in A}$ such that $n_{\alpha+1}>n_\alpha$ and $m_{\alpha+1}<m_\alpha$. The idea is explained in the picture. Obviously the sequence may be bounded or unbounded on each side.

  
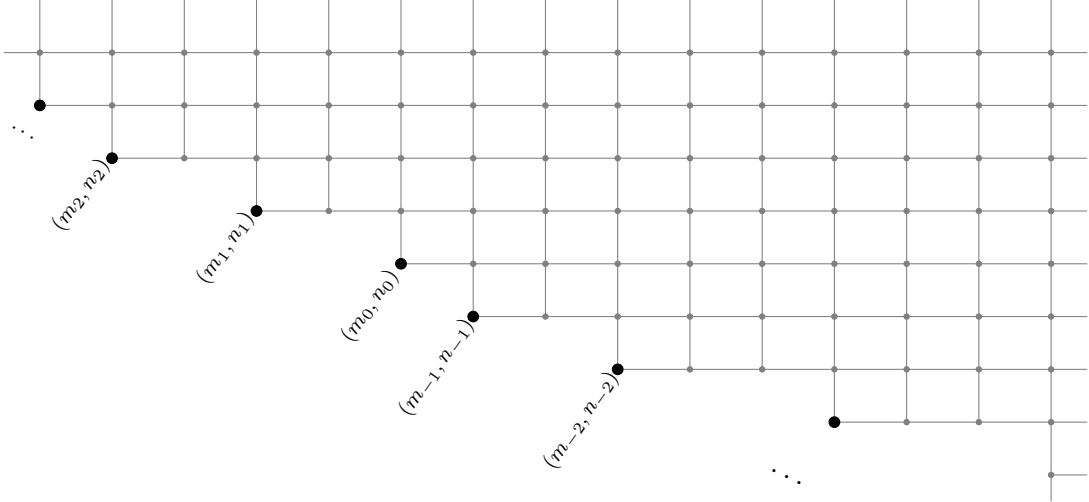
\begin{tikzpicture}
[yscale=0.7, xscale=0.95, auto,
kropka/.style={draw=black!50, circle, inner sep=0pt, minimum size=2pt, fill=black!50},
kropa/.style={draw=black, circle, inner sep=0pt, minimum size=4pt, fill=black!100},
kolo/.style={draw=black, circle, inner sep=0pt, minimum size=8pt, semithick},
kwadrat/.style={draw=black, rectangle, inner sep=0pt,  minimum size=10pt, semithick},
romb/.style={draw=black, diamond, inner sep=0pt, minimum size=8pt, semithick},
trojkat/.style ={draw=black, regular polygon, regular polygon sides=3, inner sep=0pt, minimum size=8pt, semithick}]
\draw [black!50] (9.5,-4) -- (9,-4);
\draw [black!50] (9.5,-3) -- (6,-3);
\draw [black!50] (9.5,-1) -- (3,-1);
\draw [black!50] (9.5,0) -- (1,0);
\draw [black!50] (9.5,1) -- (0,1);
\draw [black!50] (9.5,2) -- (-2,2);
\draw [black!50] (9.5,3) -- (-5,3);
\draw [black!50] (9.5,4) -- (-5.5,4);

\draw [black!50] (9,-4.5) -- (9,5);
\draw [black!50] (8,-3) -- (8,5);
\draw [black!50] (7,-3) -- (7,5);
\draw [black!50] (6,-3) -- (6,5);
\draw [black!50] (5,-2) -- (5,5);
\draw [black!50] (4,-2) -- (4,5);
\draw [black!50] (3,-1) -- (3,5);
\draw [black!50] (2,-1) -- (2,5);
\draw [black!50] (1,0) -- (1,5);
\draw [black!50] (0,1) -- (0,5);
\draw [black!50] (-1,1) -- (-1,5);
\draw [black!50] (-2,2) -- (-2,5);
\draw [black!50] (-3,2) -- (-3,5);
\draw [black!50] (-5,3) -- (-5,5);

\path (9,-1) node [kropka] {}(8,-1) node [kropka] {}(7,-1) node [kropka] {}(6,-1) node [kropka] {} (5,-1) node [kropka] {} (4,-1) node [kropka] {};
\path (9,0) node [kropka] {}(8,0) node [kropka] {}(7,0) node [kropka] {}(6,0) node [kropka] {} (5,0) node [kropka] {} (4,0) node [kropka] {} (3,0) node [kropka] {} (2,0) node [kropka] {};
\path (9,1) node [kropka] {}(8,1) node [kropka] {}(7,1) node [kropka] {}(6,1) node [kropka] {} (5,1) node [kropka] {} (4,1) node [kropka] {} (3,1) node [kropka] {} (2,1) node [kropka] {}
(1,1) node [kropka] {};
\path (9,2) node [kropka] {}(8,2) node [kropka] {}(7,2) node [kropka] {}(6,2) node [kropka] {} (5,2) node [kropka] {} (4,2) node [kropka] {} (3,2) node [kropka] {} (2,2) node [kropka] {}
(1,2) node [kropka] {} (0,2) node [kropka] {} (-1,2) node [kropka] {};
\path (9,3) node [kropka] {}(8,3) node [kropka] {}(7,3) node [kropka] {}(6,3) node [kropka] {} (5,3) node [kropka] {} (4,3) node [kropka] {} (3,3) node [kropka] {} (2,3) node [kropka] {}
(1,3) node [kropka] {} (0,3) node [kropka] {} (-1,3) node [kropka] {} (-2,3) node [kropka] {} (-3,3) node [kropka] {}(-5,3) node [kropa] {};
\path (9,4) node [kropka] {}(8,4) node [kropka] {}(7,4) node [kropka] {}(6,4) node [kropka] {} (5,4) node [kropka] {} (4,4) node [kropka] {} (3,4) node [kropka] {} (2,4) node [kropka] {}
(1,4) node [kropka] {} (0,4) node [kropka] {} (-1,4) node [kropka] {} (-2,4) node [kropka] {} (-3,4) node [kropka] {}(-5,4) node [kropka] {};
\path (9,-2) node [kropka]{}(8,-2) node [kropka]{}(7,-2) node [kropka]{};
\path (9,-3) node [kropka]{}(8,-3) node [kropka]{} (7,-3) node [kropka]{}(6,-3) node [kropa]{};
\path (9,-4) node [kropka]{};
\draw [black!50]  (9.5,-2) -- (3,-2) -- (3,-1) -- (1,-1) -- (1,0) -- (0,0) --  (0,1) -- (-2,1) -- (-2,2) -- (-4,2) -- (-4,5);

\path (6,-2) node [kropka] {}  ++(-1,0) node [kropka]  {}
        ++(-1,0) node [kropka] {} ++(-1,0) node [kropa] {}
        ++(0,1) node [kropka] {} ++(-1,0) node [kropka] {}
        ++(-1,0) node [kropa] {} ++(0,1) node [kropka] {}
        ++(-1,0) node [kropa] {}  ++(0,1) node [kropka] {}
        ++(-1,0) node [kropka] {}  ++(-1,0) node [kropa] {}
        ++(0,1) node [kropka] {} ++(-1,0) node [kropka] {} ++(-1,0) node [kropa] {}
        ++(0,1) node [kropka] {} ++(0,1) node [kropka] {};
{\tiny
\node at (-5.4,2.6) {\begin{rotate}{-30}$\dots$\end{rotate}};
 \node at (-4.7,0.6) {\begin{rotate}{55}$(m_2,n_2)$\end{rotate}};
 \node at (-2.7,-0.4) {\begin{rotate}{55}$(m_1,n_1)$\end{rotate}};
  \node at (-0.7,-1.5) {\begin{rotate}{55}$(m_0,n_0)$\end{rotate}};
  \node at (0.1,-2.9) {\begin{rotate}{55}$(m_{-1},n_{-1})$\end{rotate}};
  \node at (2.1,-3.9) {\begin{rotate}{55}$(m_{-2},n_{-2})$\end{rotate}};}
\node at (5.1,-3.9) {\begin{rotate}{-27}$\dots$\end{rotate}};
\end{tikzpicture}

 Let us denote subspaces $E_{m_\alpha,n_\alpha}:=\bigoplus_{\substack{i\ge m_\alpha\\ j\ge n_\alpha}}\mathbb{C}e_{i,j}$. Since $J=\bigcup_{\alpha\in A}(n_\alpha,m_\alpha)+\Z_+^2$ then $\M=\bigvee_{\alpha\in A}E_{m_\alpha, n_\alpha}$.  Note that $(L_w|_{E_{m_\alpha,n_\alpha}},L_z|_{E_{m_\alpha,n_\alpha}})$ are doubly commuting unilateral shifts. Therefore $E_{m_\alpha,n_\alpha}=\psi_\alpha \h^2(\T^2)$ for a unimodular function $\psi_\alpha$ (\cite{GM}, Corollary 4). Hence $e_{m_\alpha,n_\alpha}=\psi_\alpha$ and $e_{m_\alpha,n_{\beta}}=w^{m_\alpha-m_{\beta}}\psi_{\beta}=z^{n_{\beta}-n_\alpha}\psi_\alpha$ for any $\alpha, \beta\in A$. Consequently $w^{m_\beta-m_\alpha}\overline{z}^{(n_\beta-n_\alpha)}\psi_\alpha = \psi_\beta$ so  $\psi:=\overline{w}^{m_\alpha}\overline{z}^{n_\alpha}\psi_\alpha$  do not depend on the choice of $\alpha$. Eventually, $E_{m_\alpha,n_\alpha}=\psi w^{m_\alpha}z^{n_\alpha}\h^2(\T^2)$ and consequently $\M = \psi\h_J$.

 Since the space of a minimal unitary extension of $L_w|_{\M_J}, L_z|_{\M_J}$ equals to $L^2(\T^2)$ then the space of a minimal unitary extension of $L_w|_{\psi\M_J}, L_z|_{\psi\M_J}$ equals to $\psi L^2(\T^2)=L^2(\T^2)$.
 \end{proof}
Note that for a diagram $\mathbb{Z}_+\times\mathbb{Z}$ by the model in \cite{BCL} we get $\M\simeq \h^2(\T)\otimes L^2(\T)$. If we identify $L^2(\T^2)$ with $L^2(\T)\otimes L^2(\T)$ then $\M=\M_w\otimes\M_z$ where $\M_w\simeq \h^2(\T)$ and $\M_z\simeq L^2(\T)$ and $\M_w, \M_z$ are regarded as subspaces of $L^2(\T)$ spaces. Thus we get $\M=\psi_w\h^2(\T)\otimes L^2(\T)$ and $\psi(w,z)=\psi_w(w)\otimes 1$ where $\psi_w\in L^2(\T)$ is a unimodular function and $\psi\in L^2(\T^2)$ is as in Theorem \ref{diagramL}. Let us show this result more generally.

\begin{thm}\label{mainusL}
Let $\M\not=\{0\}$ be an invariant subspace of $L^2(\T^2)$. Operator $L_w|_{\M}$ is a unilateral shift and $L_z|_{\M}$ is a unitary operator if and only if
$$\M=\psi_w\h^2(\T)\otimes \chi_\delta L^2(\T),$$
where $\delta \subset\T$ is a Borel set, $\psi_w\in L^2(\T)$ is a unimodular function of variable $w$.

 Moreover, the space of a minimal unitary extension of $L_w|_{\M}, L_z|_{\M}$ equals to $\chi_{\T\times\delta}L^2(\T^2)$.
\end{thm}
\begin{proof}
It is convenient to consider operators on the space $L^2(\T)\otimes L^2(\T)$. Then $L_w= \tilde{L}_w\otimes I, L_z= I\otimes \tilde{L}_z$ where $\tilde{L}_w, \tilde{L}_z$ denotes the multiplication operators on respective $L^2(\T)$ spaces. Since $\M=\bigoplus_{i\ge 0}L_w^i\ker (L_w|_{\M})^*$ where each $L_w^i\ker (L_w|_{\M})^*$ is reducing under $L_z$ (Theorem 3.5 in \cite{BKPSLAA}) we may put $\M=\M_w\otimes \M_z$. In details, from the model in Theorem 3.1 in \cite{BCL} (see also Theorem 4.2 in  \cite{Pop}) it follows  $(L_w|_{\M}, L_z|_{\M})  \simeq (T_w\otimes I,I\otimes U)\in\B(\h^2(\T)\otimes K)$ where $K\simeq(\ker L_w|_{\M})^*$.  Note that $\h^2(\T)$ denotes a model space, not the precise Hardy subspace of $L^2(\T)$ in the considered tensor product. Thus $\tilde{L}_w|_{\M_w}$ is a unilateral shift and $\tilde{L}_z|_{\M_z}$ is a unitary operator.  Since $(\tilde{L}_z|_{\M_z})^*=P_{\M_z}\tilde{L}_z^*|_{\M_z}$ then $\tilde{L}_z|_{\M_z}$ is unitary if and only if $M_z$ reduces $\tilde{L}_z$. Thus, the result of Helson yields $\M_z=\chi_\delta L^2(\T)$  for some Borel set $\delta\subset\T$. Similarly, $\tilde{L}_w|_{M_w}$ is a unilateral shift (completely non unitary isometry) if $\M_w$ is purely invariant. Thus $\M_w=\psi_w \h^2(\T)$ for some unimodular function $\psi_w$. So, $\M=\psi_w\h^2(\T)\otimes \chi_\delta(z) L^2(\T)$.

Obviously the space of a unitary extension of $L_w|_{\M}, L_z|_{\M}$ is $L^2(\T)\otimes \chi_\delta(z) L^2(\T)\subset L^2(\T)\otimes L^2(\T)$ which is equivalent to $\chi_{\T\times\delta}L^2(\T^2)\subset L^2(\T^2)$.
\end{proof}
In $L^2(\T^2)$, unlike in $\h^2(\T^2)$, there may exist invariant subspaces where $L_w, L_z$ are generalized powers. Recall that generalized powers are defined by a unitary operator having a star cyclic vector.

\begin{rem}\label{cyclic}
By the  result of Helson, each subspace of $L^2(\T)$ reducing under $L_z$ is of the form $\M=\chi_\gamma L^2(\T)$ for some Borel set $\gamma\subset\T$. In other words, the operator of multiplication by $\chi_\gamma$ is equal to $P_{\M}$. Since $\M$ reduces $L_z$ then $P_\M$ commutes with $L_z$. Thus $\Span{L^n_z(\chi_\gamma \mathbf{1}):n\in\mathbb{Z}}=\Span{L^n_zP_\M \mathbf{1}:n\in\mathbb{Z}}=P_\M\Span{z^n:n\in\mathbb{Z}}=P_\M L^2(\T)=\M$. Concluding, any reducing subspace of
$L^2(\T)$ has a star cyclic vector $\chi_\gamma \mathbf{1}$. In fact for any proper subspace, there is a cyclic vector.\end{rem}

By the remark above any unitary part of a bilateral shift of multiplicity one may define generalized powers.
Since generalized powers are unilateral shifts, then their unitary extensions are bilateral shifts. In the following example such an extension is a proper subspace of $L^2(\T^2)$. It follows an interesting observation, that there are proper subspaces of $L^2(\T^2)$ reducing $L_w,L_z$ to bilateral shifts. Recall that a vector $x$ is wandering for a pair  $L_w, L_z$ if $L_w^iL_z^jx\perp L_w^{i'}L_z^{j'}x$ whenever $(i,j)\ne(i',j')$ for $(i,j),(i',j')\in\mathbb{Z}_+^2$. Every vector wandering for a pair generates a subspace equivalent to $\h^2(\T^2)$ and a unitary extension acts on $L^2(\T^2)$. Thus, a proper subspace of $L^2(\T^2)$ reducing under $L_w, L_z$ may not contain any vector wandering for a pair. In the following example $L_w|_\h, L_z|_\h$ are bilateral shifts, but $\h$ does not contain any vector wandering for the pair.

\begin{ex}\label{e1}
  Let $L_{w\bar z}$ denotes the operator of multiplication by $w\bar{z}$ and $H_i:=\bigvee\{w^kz^{i-k}:k\in\mathbb{Z}\}$ for $i\in\mathbb{Z}$. Note that $L^2(\T^2)=\bigoplus_{i\in\mathbb{Z}}H_i$. Let us show that $L_wH_i=L_zH_i=H_{i+1}$.
For $L_w$ it follows from the equality $L_w(w^kz^{i-k})=w^{k+1}z^{i-k}=w^{k+1}z^{i+1-(k+1)}$ valid for every $k\in\mathbb{Z}$. In fact $L_w$ establishes a unitary equivalence between $H_i$ and $H_{i+1}$.
Moreover, since $L_{w\bar z}w^kz^{i-k}=w^{k+1}z^{i-k-1}=w^{k+1}z^{i-(k+1)}$ the space $H_i$ reduces $L_{w\bar z}$ to a bilateral shift of multiplicity $1$ for every $i$. Thus $H_i=L_{w\bar z}H_i=L^*_{w\bar z}H_i$. On the other hand, $L_z=L_{w\bar wz}=L_wL_{\bar wz}=L_wL^*_{w\bar z}$. Hence $L_zH_i=L_wL^*_{w\bar z}H_i=L_wH_i=H_{i+1}$.

Let $\Ll_0\oplus\h_0 = H_0$ be a proper decomposition reducing the unitary operator $L_{w\bar z}$. Let us define $\h_i=L_w^i\h_0$ for $i\in\mathbb{Z}_+$.
Note that $\h_i\subset L_w^iH_0=H_i$ and consequently we can define $\h_+:=\bigoplus_{i\in\mathbb{Z}_+} \h_i$ which is invariant under $L_w$ and $L_w|_{\h_+}$ is a unilateral shift. Since $\h_0$ reduces $L_{w\bar z}$ it holds $\h_0=L^{*i}_{w\bar z}\h_0$ for any $i\in\mathbb{Z}$. On the other hand $L_z^i=L^i_wL^{*i}_{w\bar z}$ and consequently $\h_i=L_w^i\mathcal{H}_0=L^i_zL^{*i}_{w\bar z}\mathcal{H}_0=L^i_z\mathcal{H}_0$. Thus $\mathcal{H}_+$ is a subspace invariant under $L_z$ and $L_w$ where the operators are unilateral shifts. By Remark \ref{cyclic} operator $\mathcal{U}=L_{w\bar z}|_{\h_0}$ have a star cyclic vector. Let $J=\bigcup_{k\in\mathbb{Z}}J_0+k(1,-1)$ where $J_0=\{0\}\times\mathbb{Z}_+$. By Lemma \ref{gplemma} operators $L_w|_{\h_+},L_z|_{\h_+}$ are generalized powers given by $J$ and $\mathcal{U}$.

Extending definition of $\h_i$ for negative $k$ as powers of the adjoint we get $\h:=\bigoplus_{i\in\mathbb{Z}} \h_i$ which reduces $L_w, L_z$ to bilateral shifts. Similarly $\Ll:=\bigoplus_{i\in\mathbb{Z}} \Ll_i$, where $\Ll_i=L^i_w \Ll_0$ reduces $L_w, L_z$ to bilateral shifts. Moreover, $L^2(\T^2)=\Ll\oplus\h$.
\end{ex}

The proof of the following theorem is based on the above example.

\begin{thm}\label{maingp}
Restrictions $L_w|_\M, L_z|_\M$ are generalized powers for some invariant subspace $\M\subset L^2(\T^2)$ if and only if a unitary operator defining them as generalized powers is a unitary part of $L_z^{*n}L_w^m|_{\M_{0,0}}$ where $\M_{0,0}$ reduces $L_z^{*n}L_w^m$ to a bilateral shift of multiplicity one  and $m,n$ coincide with numbers in the periodic diagram $J=\bigcup_{k\in\mathbb{Z}}J_0+k(m,-n)$.

 Moreover, any such subspace $\M$ is determined by a periodic diagram and a Borel set $\gamma\subset \T$.

\end{thm}

\begin{proof}
 Let $\M\subset L^2(\T^2)$ be an invariant subspace such that $L_w|_\M, L_z|_\M$ are generalized powers. Recall, that by the definition of generalized powers and their properties there are: space $\h$ such that $\M=\bigoplus_{(i,j)\in J_0}\M_{i,j}$ for $\M_{i,j}=\h$, a unitary operator $\mathcal{U}\in \B(\h)$ which extension onto the whole $\M$, denoted by $U$, commutes with $L_w|_\M, L_z|_\M$ and satisfy $L_w^m=UL_z^n$ on $\M$. In conclusion,  $\mathcal{U}$ is a part of $L_z^{*n}L_w^m$ which is a bilateral shift on the whole $L^2(\T^2)$ . On the other hand, the extension of $\mathcal{U}$ to a bilateral shift generates, by the Definition \ref{gpdef} generalized powers given by the same diagram as the pair $L_w|_\M, L_z|_\M$. Obviously the space of such an extension is a subspace of $L^2(\T^2)$. Recall that a pair of generalized powers defined by a bilateral shift and some diagram $J$ is also a pair given by the same diagram $J$.    However,  there may be only simple diagrams in $L^2(\T^2)$. It is an easy observation that if a pair of generalized powers defined by a bilateral shift is a pair given by a simple diagram then the bilateral shift is of multiplicity one. Thus we have showed that the only possible pairs of generalized powers may be of the form assumed in the theorem. Any unitary part of a bilateral shift of multiplicity one is determined by some Borel subset of a circle.

Let us show that a pair of generalized powers defined by any unitary part of a bilateral shift and any periodic diagram may be realized as $L_w|_\M, L_z|_\M$.
 We start with $L_w|_{\M}, L_z|_{\M}$ given by a periodic diagram $J$. By Theorem \ref{diagramL} $\M=\psi\M_J$. It was already recalled that $L_w|_{\M}, L_z|_{\M}$ is also a pair of generalized powers given by the same diagram and a bilateral shift. For the sake of completeness of the proof we show it using Lemma \ref{gplemma}. First note that $\M=\psi\M_J=\bigoplus_{(i,j)\in J}\mathbb{C}\psi w^iz^j$ and, by the periodicity of $J$ there are numbers $n,m$ such that the operator $U:=L_z^{*n}L_w^m$ is unitary on $\M$. Let $\M_{i,j}:=\bigoplus_{k\in\mathbb{Z}} \mathbb{C}\psi w^{i+km}z^{j-kn}$ for $(i,j)\in J_0$. Since $J=\bigcup_{k\in\mathbb{Z}}J_0+k(m,-n)$ it holds ${\M}=\bigoplus_{(i,j)\in J_0}\M_{i,j}$ which is the decomposition required in Lemma \ref{gplemma}. Obviously $L_z^{*n}L_w^m|_{\M_{i,j}}$ is a bilateral shift of multiplicity one for every $(i,j)\in J_0$ which fulfills the first condition in the mentioned lemma. Since  $L_z\M_{i,j}=\M_{i,j+1}$ for $(i,j)\in J_0$ and $L_w\M_{i,j}=\M_{i+1,j}$ for $(i,j)\in J_0$ where $i\ne m-1$ then $L_w^{i'-i}L_z^{j'-j}$ is a unitary operator between subspaces $\M_{i,j}$ and $\M_{i',j'}$ - the second condition of the lemma.

   Let $\gamma\subset \T$ be an arbitrary Borel set and then $\chi_\gamma L^2(\T)$ is an arbitrary subspace reducing a bilateral shift of multiplicity one. Let  $\Ll^\gamma_{0,j_0}\subset\M_{0,j_0}$ be such a subspace reducing $L_z^{*n}L_w^m|_{\M_{0,j_0}}$ for a chosen $j_0$. For arbitrary $(i,j)\in J_0$ let $\Ll^\gamma_{i,j}=L_w^iL_z^{j-j_0}\Ll^\gamma_{0,j_0}\subset\M_{i,j}$ - recall that $L_w^iL_z^{j-j_0}$ are unitary operators between $\M_{0,j_0}$ and $\M_{i,j}$. Thus $L_z\M_{i,j}=\M_{i,j+1}$ for any $(i,j)\in J_0$ and $L_w\M_{i,j}=\M_{i+1,j}$ for $(i,j)\in J_0, i\ne m-1$. Since $\Ll^\gamma_{i,j}$ reduces $U$  it holds $L_w\Ll^\gamma_{m-1,j}=L_w^m\Ll^\gamma_{0,j}=L_z^nU\Ll^\gamma_{0,j}=L_z^n\Ll^\gamma_{0,j}=\Ll^\gamma_{0,j+n}$. Thus $\Ll^\gamma:=\bigoplus_{(i,j)\in J_0}\Ll^\gamma_{i,j}$ is invariant under $L_w, L_z$. One can check that by the definition of subspaces $\Ll^\gamma_{i,j}$ the decomposition $\bigoplus_{(i,j)\in J_0}\Ll^\gamma_{i,j}$ fulfills the conditions of Lemma \ref{gplemma}. Thus $L_w|_{\Ll^\gamma}, L_z|_{\Ll^\gamma}$ is a pair of generalized powers defined by the same diagram $J$ and a unitary part of a bilateral shift $\mathcal{U}|_{\chi_{\gamma}L^2(\T)}$.
\end{proof}

 In the construction in Example \ref{e1}  we used a Borel subset $\gamma\subset\T$ to get a reducing subspace of $L^2(\T^2)$ which is equal $\chi_\Delta L^2(\T^2)$ for some Borel set $\Delta\subset\T^2$. Similarly, the unitary extension of a pair constructed in Theorem \ref{maingp} using the set $\gamma$ is equal to $\chi_\Delta L^2(\T^2)$. Let us investigate the connection between $\gamma$ and $\Delta$.
 \begin{prop}\label{gpprop}
 Let $\M$ be an invariant subspace where $L_w|_\M, L_z|_\M$ are generalized powers defined by a diagram $J=\bigcup_{k\in\mathbb{Z}}J_0+k(m,-n)$ and a Borel set $\gamma\subset \T$. Then the space of a unitary extension of $L_w|_\M, L_z|_\M$ is equal to $\chi_\Delta L^2(\T^2)$ where $\Delta=\omega^{-1}(\gamma)$ with $\omega:\T^2\ni(w,z)\to w^m\bar z^n\in \T$. \end{prop}
 \begin{proof}
   Let us start with a construction of a unitary extension. By the proof of Theorem \ref{maingp} if we extend a respective unitary operator to a bilateral shift, then we get a pair given by a diagram $J$ which by Theorem \ref{diagramL} is of the form $\psi\M_J$ for some unimodular function $\psi$. Let $\h_{i,j}=\Span{\psi w^{i+km}z^{j-kn}:k\in\mathbb{Z}}$ for $i=0,\dots,m-1, j\in\mathbb{Z}$. Since  $L_{\bar{z}^nw^m}|_{\h_{i,j}}$ is a bilateral shift of multiplicity one we can denote   a subspace reducing the bilateral shift $L_{\bar{z}^nw^m}|_{\h_{0,0}}$ equivalent to $\chi_\gamma L^2(\T)$ by $\Ll^\gamma_{0,0}$. Moreover, let $\Ll^\gamma_{i,j}:=L_w^iL_z^j\Ll^\gamma_{0,0}$ for $i=0,\dots,m-1, j\in\mathbb{Z}$. Note that $\h_{i,j}$ as well as $\Ll_{i,j}$ are pairwise orthogonal.
        Since $\Ll^\gamma_{i,j}=L_{w^m\bar{z}^n}\Ll^\gamma_{i,j}$, we have $L_w\Ll^\gamma_{m-1,j}=L_w^m\Ll^\gamma_{0,j}=L_z^nL_{\bar{z}}^nL_w^m\Ll^\gamma_{0,j}=L_z^nL_{\bar{z}^nw^m}\Ll^\gamma_{0,j}=L_z^n\Ll^\gamma_{0,j}=\Ll^\gamma_{0,j+n}$.
       The above equality and the definition of $\Ll^\gamma_{i,j}$ subspaces implies that $\Ll^\gamma:=\bigoplus\limits_{\substack{i=0,\dots,m-1\\j\in\mathbb{Z}}}\Ll_{i,j}$ is a reducing subspace under both $L_w, L_z$. Note that this is the same construction as in Theorem \ref{maingp} but for the set of indices $\{0,\dots,m-1\}\times\mathbb{Z}$ which properly contains $J_0$. Consequently $\M=\bigoplus_{(i,j)\in J_0}\Ll_{i,j}$ and $\Ll^\gamma$ is the space of the minimal unitary extension of $L_w|_\M, L_z|_\M$.

  Let $\chi_\gamma=\sum_{k\in\mathbb{Z}}\alpha_kx^k$ be the Fourier expansion. Recall, that it is supposed to be a reducing subspace of $L_{\bar{z}^nw^m}|_{\h_{0,0}}$. Thus $x\simeq \bar{z}^nw^m$ suggest to define a function in $L^2(\T^2)$ as $f(w,z):=\sum_{(k,l)\in\mathbb{Z}^2}\alpha_{k,l}w^kz^l$ where
   $$\alpha_{k,l}=\left\{
                      \begin{array}{ll}
                        \alpha_\frac{k}{m}, & \text{ for } \frac{k}{m}=-\frac{l}{n}\in\mathbb{Z} \\
                        0, & \text{ for remaining } (i,j)
                      \end{array}
                    \right..$$
 Since $f(w,z)=\chi_\gamma(\omega(w,z))$, we get $f(w,z)=\chi_{\omega^{-1}(\gamma)}$ and $\Ll^\gamma_{0,0}=P_{\chi_{\omega^{-1}(\gamma)}L^2(\T^2)}\h_{0,0}$. Moreover, $\Ll^\gamma_{i,j}:=L_w^iL_z^j\Ll^\gamma_{0,0}=L_w^iL_z^jP_{\chi_{\omega^{-1}(\gamma)}L^2(\T^2)}\h_{0,0}\subset \chi_{\omega^{-1}(\gamma)} L^2(\T^2)$ because $\chi_{\omega^{-1}(\gamma)} L^2(\T^2)$ is reducing under $L_w, L_z$. Thus $\Ll^\gamma\subset \chi_{\omega^{-1}(\gamma)} L^2(\T^2)$. On the other hand, $\h=\Ll^\gamma\oplus\Ll^{\T\setminus\gamma}$. Similar arguments for the set $\T\setminus\gamma$ leads to the conclusion $\h\ominus\Ll^\gamma=\Ll^{\T\setminus\gamma}\subset (1-\chi_{\omega^{-1}(\gamma)})L^2(\T^2)$. Eventually, $\Ll^\gamma=\chi_{\omega^{-1}(\gamma)} L^2(\T^2)$.

 Note that by Remark \ref{periodre}, the set $\gamma_i$ depends on the choice of a period. However, the choice of a period determines numbers $m,n$, and so the polynomial $\omega$. Thus for different $\gamma$ via different $\omega$ the set $\Delta$ is supposed to be the same.
\end{proof}
  Let us take a closer look to the set $\omega^{-1}(\gamma)$.
  \begin{rem}\label{gprem}
  Fix $\gamma_0\in\gamma$. Equation $w^m\bar z^n=\gamma_0$ for a fixed $z$ has $m$ solutions and for a fixed $w$ has $n$ solutions. Identify $(w,z)\in\T^2$ with $(\Arg(w),\Arg(z))\in[0,2\pi)^2$. Then $\Arg (w)=\frac{\Arg (\gamma_0)+n\Arg (z)+2k\pi}{m}$ and $\omega^{-1}(\gamma_0)$ is represented on $[0,2\pi)^2$ as sections inclined at an angle $\atan(\frac{n}{m})$. Consequently, the solution of $w^m\bar z^n=\gamma_0$ is a line winding on $\T^2$ finite times. In conclusion, the set $\omega^{-1}(\gamma)$ is a sum of stripes parallel to each other but never to $\Arg(w)$ or $\Arg(z)$ axis. The picture illustrate $\omega^{-1}(\gamma_0)$ for $m=5, n=3$:

 
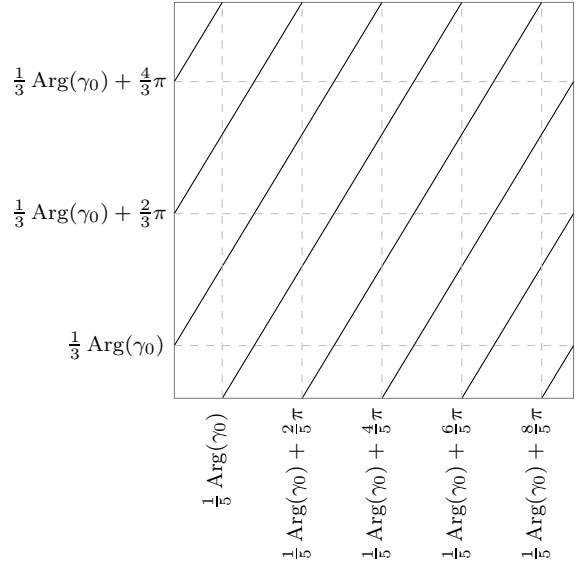
\begin{center}
  \begin{tikzpicture}
[yscale=0.7, xscale=0.7, auto]
\draw [black!50] (0,0) -- ++(7.5,0) -- ++(0,7.5) -- ++ (-7.5,0) -- ++ (0,-7.5);
\draw [black!100] (0,6) -- (0.9,7.5);\draw [black!25,dashed] (0,6) -- (7.5,6);\node at (-1.6,6) {\tiny$\frac13\Arg(\gamma_0)+\frac43\pi$};
\draw [black!100] (0,3.5) -- (2.4,7.5);\draw [black!25,dashed] (0,3.5) -- (7.5,3.5);\node at (-1.6,3.5) {\tiny$\frac13\Arg(\gamma_0)+\frac23\pi$};
\draw [black!100] (0,1) -- (3.9,7.5);\draw [black!25,dashed] (0,1) -- (7.5,1);\node at (-1.1,1) {\tiny$\frac13\Arg(\gamma_0)$};
\draw [black!100] (0.9,0) -- (5.4,7.5);\draw [black!25,dashed] (0.9,0) -- (0.9,7.5);
\node at (6.9,-3.1) {\begin{rotate}{90}\tiny$\frac15\Arg(\gamma_0)+\frac{8}{5}\pi$\end{rotate}};
\draw [black!100] (2.4,0) -- (6.9,7.5);\draw [black!25,dashed] (2.4,0) -- (2.4,7.5);
\node at (5.4,-3.1) {\begin{rotate}{90}\tiny$\frac15\Arg(\gamma_0)+\frac{6}{5}\pi$\end{rotate}};
\draw [black!100] (3.9,0) -- (7.5,6);\draw [black!25,dashed] (3.9,0) -- (3.9,7.5);
\node at (3.9,-3.1) {\begin{rotate}{90}\tiny$\frac15\Arg(\gamma_0)+\frac{4}{5}\pi$\end{rotate}};
\draw [black!100] (5.4,0) -- (7.5,3.5);\draw [black!25,dashed] (5.4,0) -- (5.4,7.5);
\node at (2.4,-3.1) {\begin{rotate}{90}\tiny$\frac15\Arg(\gamma_0)+\frac{2}{5}\pi$\end{rotate}};
\draw [black!100] (6.9,0) -- (7.5,1);\draw [black!25,dashed] (6.9,0) -- (6.9,7.5);
\node at (0.9,-2.1) {\begin{rotate}{90}\tiny$\frac15\Arg(\gamma_0)$\end{rotate}};
\end{tikzpicture}
\end{center}

 \end{rem}

 It is known that the class of compatible pairs extends the class of doubly commuting pairs. Let us point some relation of generalized powers with pairs consisting of a unitary operator and a unilateral shift. By Remark \ref{gprem}
 a unitary extension of a pair of generalized powers is $\chi_\Delta L^2(\T^2)$ where $\Delta$ can be described as stripes inclined at a nonzero angle to any of axes (any angle in $\atan(\mathbb{Q}_+)$). In the cases $\h_{us}, \h_{su}$ the set $\Delta= \delta\times\T, \Delta=\T\times \delta$ respectively which are sets inclined at zero angle to one of axes. Thus the cases $\h_{us}, \h_{su}$ appeared to be border cases of generalized powers. In Definition \ref{gpdef} numbers $m, n$ are assumed to be positive. However, if we let one of them to be zero, we get the following relation.
 \begin{cor}
 A unitary operator and a unilateral shift fulfills Definition \ref{gpdef} for $m=0, n=1$. A unilateral shift and a unitary operator fulfills Definition \ref{gpdef} for $m=1, n=0$.

  Indeed, by Theorem \ref{diagramL} a respective unitary operator is a part of a bilateral shift of multiplicity one. Thus, by Remark \ref{cyclic} it has a star cyclic vector.
  \end{cor}
  Let us now formulate the final result.
\begin{thm}\label{mainL}
Let $\M\not=\{0\}$ be an invariant subspace of $L^2(\T^2)$. If the pair $(L_w|_{\M},L_z|_{\M})$ is compatible then $\M$ has one of the following forms:
 \begin{enumerate}
 \item $$\M=\chi_\Theta L^2(\T^2)\oplus\left(\chi_\delta L^2(\T) \otimes \psi_z\h^2(\T)\right)$$
were $\delta \subset\T,\;\Theta\subset\T^2\setminus(\delta\times\T)$ are Borel sets, $\psi_z\in L^2(\T)$ is a unimodular function of variable $z$,
\item $$\M=\chi_\Theta L^2(\T^2)\oplus\left(\psi_w\h^2(\T)\otimes \chi_\delta L^2(\T)\right) $$
were $\delta \subset\T,\;\Theta\subset\T^2\setminus(\T\times\delta)$ are Borel sets, $\psi_w\in L^2(\T)$ is a unimodular function of variable $w$,
\item $$\M=\psi \M_J ,$$
where $\psi\in L^2(\T^2)$ is a unimodular function, $J\subset \Z^2$ is a diagram and $\M_J=\{w^iz^j:(i,j)\in J\}$,
\item $$\M=\chi_\Theta L^2(\T^2) \oplus \bigoplus_i\Ll^{\gamma_i}_{J^i},$$
where there are positive integers $m, n$ such that  $L_w|_{\Ll^{\gamma_i}_{J^i}}, L_z|_{\Ll^{\gamma_i}_{J^i}}$ are generalized powers defined by a periodic diagram $J^i=\bigcup_{k\in\mathbb{Z}}J_0^i+kl_i(m,-n)$ and a unitary part of a bilateral shift $\chi_{\gamma_i}L^2(\T)$ where $\gamma_i\subset \T$ is a Borel set for each $i$. Moreover, $\omega_i^{-1}\gamma_i$ are pairwise almost disjoint and $\Theta \subset \T^2\setminus \bigcup_i\omega_i^{-1}(\gamma_i)$ where $\omega_i:\T^2\ni(w,z)\to \left(\bar{z}^{n}w^{m}\right)^{l_i}\in\T$.
\end{enumerate}
\end{thm}
Note that subspaces where the operators are a pair of doubly commuting unilateral shifts is the case $J=\mathbb{Z}_+^2$ in \textit{(3)}.
\begin{proof}
By Theorem \ref{decomposition_c_cnc}
  $\M_{uu}\oplus\M_{us}\oplus\M_{su}\oplus\M_{Hardy}\oplus \M_{di} \oplus \M_{gp} \oplus\M_{cnc}$, where $\M_{di}=\M_{Hardy}\oplus\M_d$ and by compatibility $\M_{cnc}=\{0\}$. Denote $\M_{uu}=\chi_\Theta L^2(\T^2)$ and the spaces of minimal unitary extensions of respective restrictions by $\chi_{\Delta_\iota} L^2(\T^2)$ for $\iota=us, su, di, gp$. By Remark \ref{KbotK} sets $\Theta,\Delta_{us},\Delta_{su},\Delta_{di},\Delta_{gp}$ are pairwise almost disjoint.

  If $\M_{di}\ne\{0\}$ then by Theorem \ref{diagramL} $\M=\M_{di}=\psi \M_J$ and $\Delta_{di}=\T^2$. Thus $\M_{uu}=\M_{us}=\M_{su}=\M_{gp}=\{0\}$.

By the descriptions of $\Delta_{us}, \Delta_{su}$ in Theorem \ref{mainusL} and $\Delta_{gp}$ (precisely parts of $\Delta_{gp}$) in Remark \ref{gprem} sets $\Delta_{us},\Delta_{su},\Delta_{gp}$ are disjoint only if at most one of them is of positive measure. Hence, there left three possibilities: $\M=\M_{uu}\oplus\M_{us},\;\M=\M_{uu}\oplus\M_{su},\;\M=\M_{uu}\oplus\M_{gp}$.

  The cases $\M=\M_{uu}\oplus\M_{us},\;\M=\M_{uu}\oplus\M_{su}$ follow from Theorem \ref{mainusL}.

  The last case is $\M=\M_{uu}\oplus\M_{gp}$. By Theorem \ref{decomposition_c_cnc}, the restrictions to the space $\M_{gp}$ can be decomposed into generalized powers. So, $\M_{gp}=\bigoplus_i\Ll^{\gamma_i}_{J_i}$ where the restrictions to each $\Ll^{\gamma_i}_{J^i}$ are generalized powers. Denote by $\chi_{\Delta_i}L^2(\T^2)$ the spaces of their minimal unitary extensions. By Remark \ref{KbotK} $\Delta_i$ are pairwise disjoint. On the other hand, by the description of $\Delta_i$ following from Remark \ref{gprem} it is possible only when $\frac{n_i}{m_i}$ equals for all $i$. Thus, there are positive integers $m,n$ and a sequence of positive integers $l_i$ such that $J^i=\bigcup_{k\in\mathbb{Z}}J_0^i+kl_i(m,-n)$.  By Theorem \ref{maingp} a unitary operator defining $\Ll^{\gamma_i}_{J^i}$ is given by some Borel set $\gamma_i\subset \T$. By Proposition \ref{gpprop} we have $\Delta_i=\omega^{-1}(\gamma_i)$.  Thus $\Theta \subset \T^2\setminus \bigcup_i\Delta_i=\T^2\setminus\bigcup_i\omega_i^{-1}(\gamma_i)$.
  \end{proof}
It is difficult to compare directly sets $\gamma_i$ related to the decomposition of $\M_{gp}$. Indeed, they all describe unitary parts of bilateral shifts of multiplicity one, but  of different bilateral shifts. However, it turns out, that if the sequence $\{l_i\}$ is bounded then all spaces $\Ll^{\gamma_i}_{J^i}$ can be described by common numbers $m,n$. Indeed, then we can find a sequence $\{l_i'\}$ such that $l_il_i'=l$ for some $l$ and all $i$. Then, by Remark \ref{periodre} we can describe $\Ll^{\gamma_i}_{J^i}$ by diagrams $J^i=\bigcup_{k\in\mathbb{Z}}J^i_0+k(lm,-ln)$. Since functions $\omega_i$ depends only on $m_i,n_i$ which were changed to a common pair $lm,ln$ then all $\omega_i$ are equal. Consequently the sets $\Delta_i$ are disjoint if the new sets $\gamma'_i$ are disjoint.

\section{Acknowledgment}
The third author is indebted to Michio Seto for a helpful discussion.

\end{document}